\documentclass[11pt,a4paper]{article}
\usepackage[left=27mm,right=27mm,top=25mm,bottom=26mm]{geometry}
\usepackage[T1]{fontenc}
\usepackage{amsmath,amsthm,mathtools}
\usepackage{newtxtext,newtxmath}
\usepackage{microtype,setspace,needspace,fancyhdr}
\usepackage[hidelinks]{hyperref}
\hypersetup{pdftitle={Uniform Asymptotics for Andrews-Beck Partition Statistics},pdfsubject={Uniform asymptotic formulas for the Andrews-Beck partition statistics NT(m,n) and M_omega(m,n)},pdfauthor={Renrong Mao}}
\fancypagestyle{plain}{\fancyhf{}\fancyfoot[C]{\thepage}}
\newtheorem{theorem}{Theorem}[section]
\newtheorem{lemma}[theorem]{Lemma}
\newtheorem{proposition}[theorem]{Proposition}
\newtheorem{corollary}[theorem]{Corollary}
\numberwithin{equation}{section}

\DeclareMathOperator{\crk}{crank}
\DeclareMathOperator{\sech}{sech}
\newcommand{\NT}{\operatorname{NT}}
\newcommand{\MW}{M_{\omega}}
\newcommand{\coeff}[1]{[q^{#1}]}
\title{\vspace{-1.6em}\bfseries Uniform Asymptotics for Andrews-Beck Partition Statistics}
\author{}
\date{}

\newcommand{\no}{\nonumber}
\author{Renrong Mao}
\begin{document}
\maketitle
\vspace{-1cm}
\begin{center}
	Department of Mathematics,\\
	Soochow University, \\
	Suzhou, 215006,
	People's Republic of China\\[6pt]

	Email: rrmao@suda.edu.cn

\end{center}
\noindent\textbf{Abstract:}
Let $\NT(m,n)$ be the total number of parts in the partitions of $n$ with rank $m$, and let $\MW(m,n)$ be the total number of ones in the partitions of $n$ with crank $m$. Using the uniform generalized-rank estimate of Liu and Zhou, we establish asymptotic formulas for $NT(m,n)$ and $\MW(m,n)$. Put $\beta_n=\pi/\sqrt{6n}$ and $L_n=\log(1/\beta_n)$. For every fixed $A>0$, uniformly for $|m|\le A\sqrt n$, we prove
\begin{align*}
\NT(m,n)&=\frac{p(n)L_n}{4}\sech^2\!\left(\frac{\beta_nm}{2}\right)
+O_A\!\bigl(p(n)\log L_n\bigr)
\intertext{
and}
\MW(m,n)&=\frac{p(n)}4\log(1+e^{-\beta_nm})
\sech^2\!\left(\frac{\beta_nm}{2}\right)
+O_A\!\bigl(p(n)\sqrt{\beta_nL_n}\bigr).
\end{align*}



\section{Introduction}
	A partition of a positive integer $n$ is a sequence of non-increasing positive integers whose sum equals $n$ and
$p(n)$ is defined to be the number of partitions of $n$ while $p(0):=1$.
Ramanujan \cite{ra1} found the famous congruences:
\begin{align}
	p(5n+4)&\equiv 0\pmod 5\label {1i1}\\
	p(7n+5)&\equiv 0\pmod 7\label {1i2}\\
	p(11n+6)&\equiv 0\pmod{11}\label {1i3}.
\end{align}

To provide combinatorial interpretations for Ramanujan's congruences, Dyson \cite{Dyson1944} defined the rank of a partition to be the largest part minus the number of parts.
Let $N(m, n)$ denote the number of partitions of $n$ with rank $m$ and 
\begin{align}
	N(m,k,n):=\sum_{t\equiv m\pmod{k}}N(t,n).\no
\end{align} Dyson conjectured that
\begin{align}
	N(m,5,5n+4)&=\frac{p(5n+4)}{5}\label{1i4}
	\intertext{for all $0\leq m\leq 4$ and for $0\leq j\leq 6$}
	N(j,7,7n+5)&=\frac{p(7n+5)}{7}\label{1i5} .
\end{align}
Congruences \eqref{1i1} and \eqref{1i2} follow immediately from \eqref{1i4} and \eqref{1i5}, respectively.
Equations \eqref{1i4} and \eqref{1i5} were first proved by Atkin and Swinnerton-Dyer \cite{ASD}.

Dyson also noted that ranks of partitions could not explain \eqref{1i3}. Thus, he conjectured the existence of another partition statistic which he called ``crank" that would
provide combinatorial interpretations of \eqref{1i1}-\eqref{1i3}. After forty-four years,
based on the work of Garvan \cite{gar},  Andrews and Garvan \cite{AG1988} finally found 
the crank of
partitions.	Write a partition of $n\ge1$ as
$\lambda=(\lambda_1,\ldots,\lambda_k)$, where
$\lambda_1\ge\cdots\ge\lambda_k\ge1$ and
$|\lambda|=\sum_{i=1}^{k}\lambda_i=n$. Set
\[
\ell(\lambda):=k,\qquad
\omega(\lambda):=\#\{i:\lambda_i=1\},\qquad
\mu(\lambda):=\#\{i:\lambda_i>\omega(\lambda)\}.
\]
Thus $\ell(\lambda)$ is the number of parts and $\omega(\lambda)$ is the number of ones. The crank of $\lambda$ is defined by
\begin{align}\label{eq:rank-crank}
	\crk(\lambda):=
	\begin{cases}
		\lambda_1,&\omega(\lambda)=0,\\
		\mu(\lambda)-\omega(\lambda),&\omega(\lambda)>0.
	\end{cases}
\end{align}
Let $M(m,n)$ count the partitions of $n$ with crank $m$ and 
\begin{align}
	M(m,k,n):=\sum_{t\equiv m\pmod{k}}M(t,n).\no
\end{align}
Then the results in \cite{AG1988,gar} imply that the crank provides combinatorial interpretations of all three 
congruences \eqref{1i1}-\eqref{1i3}, i.e. 
\begin{align*}
	M(m,5,5n+4)&=\frac{p(5n+4)}{5},\\
	M(j,7,7n+5)&=\frac{p(7n+5)}{7},\\
	M(s,11,11n+6)&=\frac{p(11n+6)}{11},
\end{align*}
for all $0\leq m\leq 4, 0\leq j\leq 6, 0\leq s<11$.

Ranks and cranks of partitions are widely studied by many authors.  Dyson \cite{Dyson1989} made the following conjecture
\begin{equation}\label{eq:Dyson-profile}
	M(m,n)\sim\frac{\beta_np(n)}4
	\sech^2\!\left(\frac{\beta_nm}{2}\right),
	\qquad \beta_n:=\frac{\pi}{\sqrt{6n}}.
\end{equation}
Bringmann and Dousse \cite{BD} proved \eqref{eq:Dyson-profile} over a growing range of $m$ by extending Wright's circle method to suitable inverse theta functions. Dousse and Mertens \cite{DM} established the corresponding asymptotic formula for ranks, using the analytic structure of the rank generating function. More recently, Liu and Zhou \cite{LZ} obtained uniform estimates for generalized rank statistics \cite{krank} that include both $N(m,n)$ and $M(m,n)$.
\begin{theorem}(\cite[Theorem~1.3]{LZ})\label{thlz}
	For fixed $k$ and $|m|=o(n^{3/4})$, there holds
\begin{equation}\label{eq:LZ}
	\frac{N_k(m,n)}{p(n)}
	=\frac{\beta_n}{4}\sech^2\!\left(\frac{\beta_nm}{2}\right)
	\left(1+O_k\!\left(\frac{n+m^2}{n^{3/2}}\right)\right).
\end{equation}
\end{theorem}

For the combinatorial interpretation of $N_k(m,n)$, the reader is referred to \cite[Theorem 1.12]{krank}.
We remark that $N_1(m,n)=M(m,n)$ and $N_2(m,n)=N(m,n)$. In this paper,  applying Theorem \ref{thlz}, we obtain asymptotic formulas for Andrews-Beck partition statistics.

 Andrews \cite{AndrewsBeck} established congruences (first conjectured by Beck) for the total number parts functions associated to
ranks of partitions. Let $NT(m,n)$ denote the total number of parts in the partitions of $n$ with rank $m$ and 
\begin{align}
	NT(m,k,n):=\sum_{t\equiv m\pmod{k}}NT(t,n).\no
\end{align}
Applying  Atkin and Swinnerton-Dyer's results on rank differences modulo $5$ and $7$, Andrews obtained
\begin{align} \label{ntmod5}
	&NT(1, 5, 5n+i) - NT(4,5,5n+i) \nonumber\\&+ 2NT(2,5,5n+i) - 2NT(3,5,5n+i) \nonumber\\&\equiv 0 \pmod{5},\no
	\intertext{ for $i=1$ or $4$ and for $i=1$ or $5$,}
	&N T(1,7,7 n+i)-N T(6,7,7 n+i)\nonumber \\
	&+N T(2,7,7 n+i)-N T(5,7,7 n+i)\nonumber \\
	&-N T(3,7,7 n+i)+N T(4,7,7 n+i) \nonumber\\
	&\equiv 0 \pmod{7}.\no
\end{align}
Motivated by these congruences, Chern \cite{Chern2024} obtained
\begin{align} 
	&M_{\omega}(1, 5, 5n+4) - M_{\omega}(4,5,5n+4) \nonumber\\&+ 2M_{\omega}(2,5,5n+4) - 2M_{\omega}(3,5,5n+4) \nonumber\\&\equiv 0 \pmod{5},\no
\end{align}
where \begin{align}
	M_{\omega}(m,k,n):=\sum_{t\equiv m\pmod{k}}M_{\omega}(t,n)\no
\end{align} and $M_{\omega}(m,n)$ denotes 
the total number of ones in the partitions of $n$ with crank $m$. For more Andrews-Beck type congruences, see \cite{CMO,ChernII,ChernIII,hong,maoadv}. For studies on relations between Andrews-Beck partition statistics,
the reader is referred to \cite{xia,maorj, maoadv,maojc,maoxia,maoxia2}. 
In this paper, we obtain the following asymptotic formulas for $\NT(m,n)$ and $\MW(m,n)$.

\begin{theorem}\label{thm:MW}
Recall $\beta_n$ defined in \eqref{eq:Dyson-profile}.	For every fixed $A>0$, uniformly for integers $|m|\le A\sqrt n$,
	\begin{equation}\label{eq:MW-beta}
	\MW(m,n)=\frac{p(n)}4\log(1+e^{-\beta_nm})
		\sech^2\!\left(\frac{\beta_nm}2\right)
		+O_A\!\bigl(p(n)n^{-1/4}\sqrt{\log n}\bigr),
	\end{equation}
and
\begin{equation}\label{eq:NT-main}
\NT(m,n)=\frac{p(n)\log n}{8}
\sech^2\!\left(\frac{\pi m}{2\sqrt{6n}}\right)
+O_A\!\bigl(p(n)\log\log n\bigr).	\qquad	\qquad	\quad
\end{equation}
\end{theorem}

Applying Theorem \ref{thm:MW} yields the following corollary.
\begin{corollary}\label{cor:small-m}
If $m=m(n)$ is an integer sequence with $m=o(\sqrt n)$, then
\begin{align}
	\MW(m,n)&=\frac{\log2}{4}p(n)
	+O\!\left(p(n)\left[n^{-1/4}\sqrt{\log n}+\frac{|m|}{\sqrt n}\right]\right),\label{eq:MW-small}
\\
\NT(m,n)&=\frac18p(n)\log n
+O\!\left(p(n)\left[\log\log n+\frac{m^2}{n}\log n\right]\right)\label{eq:NT-small}.
\end{align}
In particular,
\begin{equation}\label{eq:small-equivalents}
\MW(m,n)\sim\frac{\log2}{4}p(n),\qquad \NT(m,n)\sim\frac18p(n)\log n
.
\end{equation}
Moreover, for every fixed $B>0$, \begin{equation}\label{eq:MW-preserved}
	\MW(m,n)=\frac{\log2}{4}p(n)
	+O_B\!\bigl(p(n)n^{-1/4}\sqrt{\log n}\bigr)
\end{equation}
uniformly for $|m|\le Bn^{1/4}\sqrt{\log n}$,
 and
\begin{equation}\label{eq:NT-preserved}
	\NT(m,n)=\frac18p(n)\log n+O_B\!\bigl(p(n)\log\log n\bigr)
\end{equation}
uniformly for $|m|\le B\sqrt{n\log\log n/\log n}$.
\end{corollary}
\begin{proof}
As $x\to0$,
\begin{equation}\no
\sech^2(x/2)=1+O(x^2),\qquad
\frac14\log(1+e^{-x})\sech^2(x/2)
=\frac{\log2}{4}-\frac{x}{8}+O(x^2).
\end{equation}
For all sufficiently large $n$, the hypothesis gives $|m|\le\sqrt n$. We may therefore apply both theorems with $A=1$ and substitute $\beta_nm=\pi m/\sqrt{6n}$ into these expansions. This proves \eqref{eq:MW-small} and \eqref{eq:NT-small}. Dividing by the respective leading terms gives \eqref{eq:small-equivalents}. Substitution of the two stated bounds on $|m|$ into the error terms proves \eqref{eq:MW-preserved}  and \eqref{eq:NT-preserved}.
\end{proof}

The paper is organized as follows.  Section~2 collects the partition-function and generalized-rank estimates.  Section~3 proves the theorem for the total number of ones at a prescribed crank, and Section~4 proves the theorem for the total number of parts at a prescribed rank.
\section{Some lemmas}

For an integer $s\ge0$, write
\begin{align*}
(a;q)_0:&=1,\\ (a;q)_s:&=\prod_{r=0}^{s-1}(1-aq^r),
\\
P(q):&=\prod_{r\ge1}(1-q^r)^{-1}=\sum_{n\ge0}p(n)q^n.
\end{align*}
The infinite product and power series converge for $|q|<1$.

We need the Hardy--Ramanujan asymptotic formula \cite{HR} in the form
\begin{equation}\label{eq:HR}
p(n)=\frac{1}{4\sqrt3\,n}
\exp\!\left(\pi\sqrt{\frac{2n}{3}}\right)
\left(1+O(n^{-1/2})\right).
\end{equation}

\Needspace{10\baselineskip}
\begin{lemma}\label{lem:partition}
 There are constants $C, c>0$ such that, for all sufficiently large integers $n$,
\begin{align}
p(n-h)&\le Cp(n)e^{-\beta_nh}
&&\textrm{for} \quad 0\le h\le n/2,\label{eq:ratio-near}\\
p(n-h)&\le Cp(n)e^{-c\sqrt n}
&&\textrm{for} \quad  n/2<h\le n,\label{eq:ratio-far}\\
e^{\beta_nn}P(e^{-\beta_n})&\le Cnp(n).\label{eq:P-bound}
\end{align}
\end{lemma}
\begin{proof}
Put $b_0=\pi\sqrt{2/3}$, so that $\beta_n=b_0/(2\sqrt n)$. For $0\le h\le n/2$, formula \eqref{eq:HR} gives
\[
\frac{p(n-h)}{p(n)}
\le C\frac{n}{n-h}
\exp\!\bigl(b_0(\sqrt{n-h}-\sqrt n)\bigr)
\]
for some constant $C.$
Since $n/(n-h)\le2$ and
\[
\sqrt{n-h}-\sqrt n
=-\frac{h}{\sqrt{n-h}+\sqrt n}
\le-\frac{h}{2\sqrt n},
\]
we obtain \eqref{eq:ratio-near}. 

Note that the partition function is nondecreasing. Thus, for $h>n/2$, we have
\begin{align*}
p(n-h)
&\le p(\lfloor n/2\rfloor)
\le p(\lceil n/2\rceil)
=p\bigl(n-\lfloor n/2\rfloor\bigr)
\le C p(n)e^{-\beta_n\lfloor n/2\rfloor}.\qquad\qquad\textrm{(by \eqref{eq:ratio-near})}\no
\end{align*} 
This proves \eqref{eq:ratio-far}.

Since
\begin{align*}
&\log P\left(e^{-\beta_n}\right)
&=\log \prod_{r=1}^\infty \frac{1}{1-e^{-r\beta_n}}
&=\sum_{r\ge1}\sum_{s\ge1}\frac{e^{-\beta_n rs}}s
&=\sum_{s\ge1}\frac1{s\left(e^{\beta_n s}-1\right)}
&\le\frac1\beta_n\sum_{s\ge1}\frac1{s^2}
=\frac{\pi^2}{6\beta_n},
\end{align*}
 we find that
\[
e^{\beta_nn}P\left(e^{-\beta_n}\right)
\le e^{\beta_nn}e^{\frac{\pi^2}{6\beta_n}}=e^{2\beta_nn}\le Cnp(n),
\]
where the last inequality again follows from \eqref{eq:HR}. This proves \eqref{eq:P-bound}.
\end{proof}

\begin{lemma}\label{lem:counts}
	For every fixed $A>0$, uniformly for integers $|m|\le A\sqrt n$,
	\begin{align}
		N(m,n)&=\frac{\beta_n}{4}p(n)
		\sech^2\!\left(\frac{\beta_nm}{2}\right)
		\left(1+O_A(n^{-1/2})\right),\label{eq:rank-count}\\
		M(m,n)&=\frac{\beta_n}{4}p(n)
		\sech^2\!\left(\frac{\beta_nm}{2}\right)
		\left(1+O_A(n^{-1/2})\right).\label{eq:crank-count}
	\end{align}
	In particular, $N(m,n),M(m,n)\ll_A\beta_np(n)$.
\end{lemma}
\begin{proof}
The range $|m|\le A\sqrt n$ is contained in $|m|=o(n^{3/4})$, and
	\[
	\frac{n+m^2}{n^{3/2}}\le(1+A^2)n^{-1/2}.
	\]
Taking $k=2$ and $k=1$ in \eqref{eq:LZ} proves \eqref{eq:rank-count} and \eqref{eq:crank-count}, respectively.
\end{proof}

\Needspace{11\baselineskip}

\section{Proof of \eqref{eq:MW-beta}}\label{pmw}

\Needspace{13\baselineskip}
Throughout, set
\begin{align}
L_n:=\log\frac1{\beta_n},\quad
K_n:=\frac{L_n}{\beta_n},\quad H_{m,n}:=\frac{\log \left(1+e^{-\beta_nm}\right)}{\beta_n}\label{defhn}
\end{align}
with $\beta_n$ given in \eqref{eq:Dyson-profile}.

For an integer $m$, put $j_0=\max\{1,-m\}$. Let $b_{m,j}(n)$ denote the number of partitions of $n$ with crank $m$ and exactly $j\ge1$ ones. This number is zero when $j<j_0$ by \eqref{eq:rank-crank}.
We have \begin{align}\label{eq:ones-sums}
	\MW(m,n)&=\sum_{j\ge j_0}j\,b_{m,j}(n)
	\intertext{and}
	M^+(m,n)&=\sum_{j\ge j_0}b_{m,j}(n),\label{mplus}
\end{align}
where  $M^+(m,n)$ denotes the number of partitions of $n$ with crank $m$ and at least one part equal to $1$.
Let $B_m(n):=M(m,n)-M^+(m,n)$. Note that a nonempty partition with no ones has crank equal to its largest part, which is at least $2$. Then, for $m\le1, n\ge1$, 
\begin{align}
	B_m(n)&=0\no
	\intertext{and, for $m\geq2$,}
	\label{eq:no-ones-GF}
	\sum_{n\ge1}B_m(n)q^n&=\frac{q^m}{(q^2;q)_{m-1}}.
\end{align}
Furthermore, we have the following lemma.
\begin{lemma}\label{lem:no-ones}
	For every fixed $A>0$, there exists $c_A>0$ such that, uniformly for integers $2\le m\le A\sqrt n$,
	\begin{equation}\label{eq:no-ones-bound}
		B_m(n)\ll_A n\beta_np(n)e^{-c_A/\beta_n}.
	\end{equation}
	Consequently, for every fixed $D>0$, uniformly for integers $|m|\le A\sqrt n$,
	\begin{equation}\label{eq:no-ones-negligible}
		H_{m,n}B_m(n)=o_A\!\bigl(p(n)\beta_n^D\bigr),
	\end{equation}
	where $H_{m,n}$ is defined in \eqref{defhn}.
\end{lemma}
\begin{proof}
	%
	Write $\beta=\beta_n$ and $\rho=e^{-\beta}$. Nonnegativity of the coefficients and \eqref{eq:P-bound} give
	\begin{align*}
		B_m(n)
		&\le\rho^{-n}\frac{\rho^m}{\prod_{s=2}^{m}(1-\rho^s)}\\
		&=\rho^{-n}P(\rho)(1-\rho)\rho^m
		\prod_{s\ge m+1}(1-\rho^s)\\
		&\le Cnp(n)(1-\rho)\rho^m
		\exp\!\left(-\frac{\rho^{m+1}}{1-\rho}\right).
	\end{align*}
	When $m\le A\sqrt n$, we have $\beta m\le a_A:=\pi A/\sqrt6$. For $\beta\le1$,
	\[
	\rho^{m+1}\ge e^{-a_A-1},\qquad 1-\rho\le\beta.
	\]
	Thus \eqref{eq:no-ones-bound} holds with $c_A=e^{-a_A-1}$. Moreover, for $|m|\le A\sqrt n$,
	\[
	0<H_{m,n}\le\frac{\log(1+e^{a_A})}{\beta_n}.
	\]
	It follows that $H_{m,n}B_m(n)\ll_A np(n)e^{-c_A/\beta_n}$. Since $n=\pi^2/(6\beta_n^2)$ and the exponential decay dominates every fixed power, we obtain \eqref{eq:no-ones-negligible}. For $m\le1$, the assertion follows directly from $B_m(n)=0$.
\end{proof}
For a power series \( f(q) \), let \( [q^m] f(q) \) denote the coefficient of \( q^m \) in its expansion,
and we say $f_1(q)\preceq f_2(q)$ when \( [q^m] f_1(q)\le  [q^m] f_2(q) \) for all $m.$
Recall $b_{m,j}(n)$ defined at the beginning of Section \ref{pmw}. Let
\begin{align}
	F_{m,j}(q)&:=\sum_{n\ge0}b_{m,j}(n)q^n.\label{defdaf}
\end{align}
Obviously, we have $0\preceq F_{m,j}\preceq P(q)$.
\begin{lemma}\label{lem:ones-GF}
	For $j\ge j_0$, we have
	\begin{align}
		F_{m,j}(q)
		&=\frac{(1-q)q^{j^2+(m+2)j+m}}{(q;q)_j(q;q)_{j+m}}.\label{eq:ones-GF}
	\end{align}
\end{lemma}

\begin{proof}
	Suppose that a partition has exactly $j\ge1$ ones. By \eqref{eq:rank-crank}, its crank is $m$ if and only if it has exactly $r=j+m\ge0$ parts greater than $j$. The $j$ ones contribute $q^j$. The parts $2,\ldots,j$ may occur with arbitrary multiplicities and contribute
	$\prod_{s=2}^{j}(1-q^s)^{-1}=(q^2;q)_{j-1}^{-1}$. Subtracting $j+1$ from each of the $r$ parts greater than $j$ leaves a partition with at most $r$ positive parts. Thus the contribution from these $r$ parts is
	\[
	\frac{q^{(j+1)r}}{(q;q)_r}.
	\]
	These three independent choices uniquely determine the original partition, proving \begin{align*}
		F_{m,j}(q)
		&=\frac{q^{j+(j+1)(j+m)}}{(q^2;q)_{j-1}(q;q)_{j+m}},
	\end{align*}
	which gives \eqref{eq:ones-GF}.
\end{proof}
\Needspace{12\baselineskip}

\begin{lemma}\label{lem:mode}
	Let $a>0$. For sufficiently small $\beta>0$ and an integer $m$ satisfying $|\beta m|\le a$, put
	\[
	\rho=e^{-\beta},\qquad j_0=\max\{1,-m\},\qquad
	H=\frac{\log(1+e^{-\beta m})}{\beta}.
	\]
	There exists an integer $J\ge j_0$ such that, uniformly in this range,
	\begin{equation}\label{eq:mode-location}
		J=H+O_a(1),
	\end{equation}
	and, for every integer $j\ge j_0$,
	\begin{equation}\label{eq:mode-bound}
		F_{m,j}(\rho)\le F_{m,J}(\rho)
		\exp\!\bigl(-\beta|j-J|(|j-J|-1)\bigr).
	\end{equation}
	The threshold for $\beta$ and the implied constant may depend only on $a$.
\end{lemma}

\begin{proof}
	Write $A_j=F_{m,j}(\rho)>0$. By \eqref{eq:ones-GF},
	\begin{equation}\label{eq:ratio}
		R_j:=\frac{A_{j+1}}{A_j}
		=\frac{\rho^{2j+m+3}}
		{(1-\rho^{j+1})(1-\rho^{j+m+1})}.
	\end{equation}
	Hence
	\begin{equation}\label{eq:ratio-decrease}
		\frac{R_{j+1}}{R_j}
		=\rho^2\frac{(1-\rho^{j+1})(1-\rho^{j+m+1})}
		{(1-\rho^{j+2})(1-\rho^{j+m+2})}
		\le e^{-2\beta}.
	\end{equation}
	
	Regard $R_j$ as a function of the real variable $j>\max\{-1,-m-1\}$. Its logarithmic derivative is
	\[
	-2\beta-\frac{\beta\rho^{j+1}}{1-\rho^{j+1}}
	-\frac{\beta\rho^{j+m+1}}{1-\rho^{j+m+1}}<0,
	\]
	so $R_j$ is strictly decreasing. Set $b=\rho^m$ and $y=\rho^{j+1}$. The equation $R_j=1$ becomes
	\[
	b(1-\rho)y^2-(1+b)y+1=0.
	\]
	This quadratic polynomial is positive at $y=0$ and takes the values $-b\rho$ and $-\rho/b$ at $y=1$ and $y=b^{-1}$, respectively. Its smaller root therefore lies in $(0,\min\{1,b^{-1}\})$ and is given by
	\begin{equation}\label{eq:mode-root}
		y_* =\frac{(1+b)-\sqrt{(1-b)^2+4b\rho}}{2b(1-\rho)}=\frac{2}{1+b+\sqrt{(1-b)^2+4b\rho}}.
	\end{equation}
	Put
	\[
	\varepsilon:=\frac{4b(1-\rho)}{(1+b)^2},\qquad
	0<\varepsilon\le1-\rho\le\beta.
	\]
	Then
	\[
	y_* =\frac1{1+b}\cdot\frac2{1+\sqrt{1-\varepsilon}}
	=\frac1{1+b}\left(1+O(\varepsilon)\right)=\frac1{1+b}\left(1+O(\beta)\right).
	\]
	Thus \[
	\log y_*=-\log(1+b)+O(\beta).
	\]
	Let $j_*$ denote the corresponding real solution of $R_j=1$. Then 
	\begin{equation}\label{eq:real-mode}
		j_*:=-\frac{\log y_*}{\beta}-1
		=\frac{\log(1+b)}{\beta}+O(1)=H+O(1).
	\end{equation}
	
	We must also check its distance from $j_0$. If $m\ge0$, then $j_0=1$. If $m\le-1$, then $j_0=-m$ and
	\[
	H-j_0=H+m=\frac{\log(1+e^{\beta m})}{\beta}.
	\]
	Consequently, uniformly for $|\beta m|\le a$,
	\begin{equation}\label{eq:mode-endpoint}
		H-j_0\ge\frac{\log(1+e^{-a})}{\beta}-1\longrightarrow\infty
	\end{equation}
	as $\beta\to0^+$.
	Take $J=\lceil j_*\rceil$. Equations \eqref{eq:real-mode} and \eqref{eq:mode-endpoint} imply that, for sufficiently small $\beta$ depending only on $a$, we have $J-1\ge j_0$ and
	\begin{equation}\label{eq:mode-crossing}
		R_{J-1}\ge1\ge R_J
	\end{equation}
	since $R_j$ is decreasing.
	This proves \eqref{eq:mode-location} and shows that $A_J\geq A_j$ for all $j$.
	
	For an integer $t\ge0$, equations \eqref{eq:ratio-decrease} and \eqref{eq:mode-crossing} give
	\[
	\frac{A_{J+t}}{A_J}
	=\prod_{s=0}^{t-1}R_{J+s}
	\le\prod_{s=0}^{t-1}e^{-2\beta s}
	=e^{-\beta t(t-1)}.
	\]
	Similarly, when $J-t\ge j_0$,
	\[
	\frac{A_{J-t}}{A_J}
	=\prod_{s=1}^{t}R_{J-s}^{-1}
	\le\prod_{s=1}^{t}e^{-2\beta(s-1)}
	=e^{-\beta t(t-1)}.
	\]
	Combining these inequalities proves \eqref{eq:mode-bound}.
\end{proof}

\Needspace{9\baselineskip}

\begin{proposition}\label{prop:ones-concentration}
	Recall $M^+(m,n)$ defined in \eqref{mplus}. For every fixed $A>0$, uniformly for integers $|m|\le A\sqrt n$,
	\begin{equation}\label{eq:ones-concentration}
		\MW(m,n)=H_{m,n}M^+(m,n)
		+O_A\!\bigl(p(n)\sqrt{\beta_nL_n}\bigr).
	\end{equation}
\end{proposition}

\begin{proof}
	In this proof, abbreviate $\beta=\beta_n$, $L=L_n$, $H=H_{m,n}$, and $\rho=e^{-\beta}$. Set
	$
	D:=8\sqrt{\frac{L}{\beta}}.
	$	
By the triangle inequality, we have the following elementary inequality. If $u_k\ge0$ and only finitely many $u_k$ are nonzero, then for every real $t>0$,
\begin{equation}\label{eq:weighted-deviation}
	\left|\sum_kk u_k-t\sum_ku_k\right|
	\le\sum_k|k-t|u_k.
\end{equation}	
	Applying \eqref{eq:weighted-deviation}  yields
	\begin{align}
		&	\left|\MW(m,n)-HM^+(m,n)\right|\no\\&=	\left|\sum_{j\ge j_0}j\,b_{m,j}(n)-H\sum_{j\ge j_0}b_{m,j}(n)\right|\no\\&\le
		\sum_{j\ge j_0}|j-H|b_{m,j}(n)
		\no\\&=
		\sum_ {\substack{j\ge j_0\\|j-H|\le D}}|j-H|b_{m,j}(n)+\sum_ {\substack{j\ge j_0\\|j-H|> D}}|j-H|b_{m,j}(n)
		\no\\&\le D
		\sum_ {\substack{j\ge j_0\\|j-H|\le D}}b_{m,j}(n)+\sum_ {\substack{j\ge j_0\\|j-H|> D}}|j-H|b_{m,j}(n)\no\\&\le D
		M^+(m,n)+\sum_ {\substack{j\ge j_0\\|j-H|> D}}|j-H|b_{m,j}(n)\no\\	&\le DM(m,n)
		+\sum_ {\substack{j\ge j_0\\|j-H|> D}}|j-H|b_{m,j}(n)\no
		\\&=O_A\!\bigl(p(n)\sqrt{\beta L}\bigr)+\sum_ {\substack{j\ge j_0\\|j-H|> D}}|j-H|b_{m,j}(n),\label{pb1}
		%
	\end{align}
	where the last equality follows from Lemma \ref{lem:counts}.
	
	Next, we estimate $$\sum_ {\substack{j\ge j_0\\|j-H|> D}}|j-H|b_{m,j}(n).$$
	
	Apply Lemma~\ref{lem:mode} with $a=\pi A/\sqrt6$ to obtain the corresponding  index $J\geq j_0$ such that
	\begin{align*}
		J&=H+O_a(1)
		\intertext{and}
		F_{m,j}(\rho)&\le F_{m,J}(\rho)
		\exp\!\bigl(-\beta|j-J|(|j-J|-1)\bigr)
	\end{align*}
	for $j\geq j_0.$ Since $F_{m,J}\preceq P$, nonnegativity of the coefficients and Lemma~\ref{lem:partition} give, for $j\ge j_0$,
	\begin{align}
		b_{m,j}(n)
		&=[q^n]F_{m,j}(q)\no\\&\le\rho^{-n}F_{m,j}(\rho)\no\\&\le \rho^{-n} F_{m,J}(\rho)
		\exp\!\bigl(-\beta|j-J|(|j-J|-1)\bigr)\notag\\
		&\le\rho^{-n}P(\rho)e^{-\beta t(t-1)}\no\\&
		\le Cnp(n)e^{-\beta t(t-1)}\label{eq:ones-coefficient-bound}
	\end{align}
	with $t:=|j-J|.$
	As $J-H=O_A(1)$, if $|j-H|>D\to\infty$ as $n\to\infty$ and $n$ exceeds a threshold depending only on $A$, then
	\[
	t\ge D/2\ge2,\qquad
	\beta t(t-1)\ge\frac{\beta D^2}{8}=8L.
	\]
	It follows from \eqref{eq:ones-coefficient-bound} that
	\begin{equation}\no
		b_{m,j}(n)\le Cnp(n)\beta^8\qquad(|j-H|>D).
	\end{equation}
	Thus  \begin{align}
		&	\sum_{\substack{j\ge j_0\\|j-H|>D}}|j-H|b_{m,j}(n)\no\\&=	\sum_{\substack{n\ge j\ge j_0\\|j-H|>D}}|j-H|b_{m,j}(n)
		\no\\&\le Cnp(n)\beta^8	\sum_{\substack{n\ge j\ge j_0\\|j-H|>D}}|j-H|
		\no\\&	\ll n^3p(n)\beta^8 \no\\&\ll p(n)\beta^2.\label{pb}
	\end{align}
	Substituting  \eqref{pb} into \eqref{pb1} yields
	\begin{align*}
		\left|\MW(m,n)-HM^+(m,n)\right|
		&=O_A\!\bigl(p(n)\sqrt{\beta L}\bigr).
	\end{align*}
	This proves \eqref{eq:ones-concentration}.
\end{proof}

We are now in a position to prove \eqref{eq:MW-beta}.
\begin{proof}[Proof of \eqref{eq:MW-beta}.]
	For $|m|\le A\sqrt n$, Lemma~\ref{lem:no-ones}, with $D=1$ in \eqref{eq:no-ones-negligible}, and Proposition~\ref{prop:ones-concentration} give
	\begin{align*}
		\MW(m,n)&=H_{m,n}M(m,n)
		+O_A\!\bigl(p(n)\sqrt{\beta_nL_n}\bigr)\\
		&=\frac{p(n)}4\log(1+e^{-\beta_nm})\sech^2(\beta_nm/2)
		\left(1+O_A(n^{-1/2})\right)
		+O_A\!\bigl(p(n)\sqrt{\beta_nL_n}\bigr)\\
		&=\frac{p(n)}4\log(1+e^{-\beta_nm})\sech^2(\beta_nm/2)
		+O_A\!\bigl(p(n)\sqrt{\beta_nL_n}\bigr)\\
		&=\frac{p(n)}4\log(1+e^{-\beta_nm})\sech^2(\beta_nm/2)
		+O_A\!\bigl(n^{-1/4}p(n)\sqrt{\log n}\bigr).
	\end{align*}
\end{proof}

\Needspace{16\baselineskip}
\section{Proof of \eqref{eq:NT-main}}

For nonnegative integers $a,b$, let $Q_{a,b}(q)$ be the generating function for partitions with at most $a$ parts, each at most $b$.
\begin{lemma}\label{lem:rectangle}
Let $m\ge0$ be an integer. Denote by $a_{m,k}(n)$ the number of partitions of $n$ with rank $m$ and exactly $k\ge1$ parts, and set $h_{m,k}=2k+m-1$. Then
\begin{equation}\label{eq:rectangle}
a_{m,k}(n)=\coeff{n-h_{m,k}}Q_{k-1,k+m-1}(q),
\end{equation}
and
\begin{equation}\label{eq:rectangle-bound}
Q_{k-1,k+m-1}(q)\preceq(q;q)_{k-1}^{-1}\preceq P(q).
\end{equation}
Moreover, we have
\begin{align}\label{eq:rank-sums1}
N(m,n)&=\sum_{k\ge1}a_{m,k}(n)
\intertext{and}
\NT(m,n)&=\sum_{k\ge1}k\,a_{m,k}(n).\label{eq:rank-sums}
\end{align}
\end{lemma}
\begin{proof}
If a partition $\lambda$ has rank $m$ and $k$ parts, then $\lambda_1=k+m$. Remove its largest part and subtract $1$ from each of the remaining $k-1$ parts. The resulting sequence
\[
\nu_i=\lambda_{i+1}-1\qquad(1\le i\le k-1)
\]
has at most $k-1$ positive parts, each at most $k+m-1$. Conversely, pad such a partition $\nu$ with zeros to obtain $k-1$ entries. The original partition is uniquely recovered as
\[
\lambda=(k+m,\nu_1+1,\ldots,\nu_{k-1}+1).
\]
The identity $|\lambda|=h_{m,k}+|\nu|$ proves \eqref{eq:rectangle}.

By conjugation, $(q;q)_{k-1}^{-1}$ enumerates partitions with at most $k-1$ parts. This proves \eqref{eq:rectangle-bound}. 

Equations \eqref{eq:rank-sums1} and \eqref{eq:rank-sums} follow immediately from the definitions.
\end{proof}

\begin{proposition}\label{prop:rank-concentration}
For every fixed $A>0$, uniformly for integers $0\le m\le A\sqrt n$,
\begin{equation}\label{eq:rank-concentration}
\NT(m,n)=K_nN(m,n)+O_A\!\bigl(p(n)\log L_n\bigr).
\end{equation}
\end{proposition}

\begin{proof}
Write
$
\beta=\beta_n=\frac{\pi}{\sqrt{6n}}, L=L_n=\log \frac{1}{\beta_n}, K=K_n=\frac{L_n}{\beta_n},
\Delta=\frac{4\log L}{\beta},$ $h_k=2k+m-1,
a_k=a_{m,k}(n)
$ which is the number of partitions of $n$ with rank $m$ and exactly $k\ge1$ parts.
Take $n$ sufficiently large that $L-4\log L>1$.

It follows from \eqref{eq:weighted-deviation} that
\begin{align}
&	\left|\NT(m,n)-KN(m,n)\right|\no\\&= \left|\sum_{k\ge1}k\,a_k-K\sum_{k\ge1}a_k\right|\qquad\qquad\qquad\textrm{(by \eqref{eq:rank-sums1} and \eqref{eq:rank-sums})}\no\\
\no	&\le\sum_{k\ge1}|k-K|a_k\\
	&	=\sum_{k< K-\Delta}|k-K|a_k+\sum_{K-\Delta\le k\le K+\Delta}|k-K| a_k+\sum_{k>K+\Delta}|k-K| a_k\no\\
	&=:S_1+S_2+S_3.\label{s123}
\end{align}


We first consider $S_1.$ Let $\rho:=e^{-\beta}<1.$
 If $n-h_k\ge0$, coefficient nonnegativity, \eqref{eq:rectangle} and \eqref{eq:rectangle-bound} give
\begin{align}\label{ak}
a_k=\coeff{n-h_{k}}Q_{k-1,k+m-1}(q)\le\coeff{n-h_k}(q;q)_{k-1}^{-1}
\le\frac{\rho^{-(n-h_k)}}{(\rho;\rho)_{k-1}}
\le\frac{\rho^{-n}}{(\rho;\rho)_{k-1}}.
\end{align}
 If $n-h_k<0$, then $a_k=0$, and the same upper bound remains valid. The inequality $\log(1-u)\le-u$ for $0<u<1$ implies
\begin{align}
(\rho;\rho)_{k-1}^{-1}
\no&=P(\rho)\prod_{r\ge k}(1-\rho^r)
\\&=P(\rho)\exp\left\{\log \left(\prod_{r\ge k}(1-\rho^r)\right)\right\}\no
\\&=P(\rho)\exp\left\{\sum_{r\geq k}\log \left(1-\rho^r\right)\right\}\no
\\&\le P(\rho)\exp\left\{-\sum_{r\geq k}\rho^r\right\}\no\\
&= P(\rho)\exp\!\left(-\frac{\rho^k}{1-\rho}\right).\no
\end{align}
This together with \eqref{ak} gives
\begin{align}
	a_k
	\le \rho^{-n}P(\rho)\exp\!\left(-\frac{\rho^k}{1-\rho}\right)\le Cnp(n)\exp\!\left(-\frac{\rho^k}{1-\rho}\right),\no
\end{align}
where the second inequality follows from \eqref{eq:P-bound}.
For $k\le K-\Delta$, using $1-e^{-\beta}\le\beta$, we obtain
\[
\frac{\rho^k}{1-\rho}
\ge\frac{e^{-\beta(K-\Delta)}}{\beta}
=\frac{e^{-L+4\log L}}{\beta}=L^4.
\]
Thus $a_k\le Cnp(n)e^{-L^4}$, and hence
\begin{equation}\label{s1}
S_1=\sum_{k< K-\Delta}|k-K|a_k\le K\sum_{k\le K-\Delta}a_k\le CnK^2p(n)e^{-L^4}=o(p(n)).
\end{equation}
Indeed, $n=\pi^2/(6\beta^2)$ and $K=L/\beta$, so
$nK^2e^{-L^4}\ll L^2e^{4L-L^4}\to0,$ as $n\to\infty$. 


For $S_2,$ we have \begin{align}\label{s2}
S_2=\sum_{K-\Delta\le k\le K+\Delta}|k-K| a_k\le \Delta\sum_{K-\Delta\le k\le K+\Delta} a_k\le \Delta N(m,n)=O_A\!\bigl(p(n)\log L\bigr)
\end{align}
by Lemma \ref{lem:counts}.

Next, we consider $S_3.$
Note that 
\begin{align}
	S_3&=\sum_{k>K+\Delta}|k-K| a_k\no\\&\le \sum_{k>K+\Delta}k a_k\no\\&\le \sum_{\substack{k>K+\Delta\\h_k\le n/2}}kp(n-h_k)+\sum_{\substack{k>K+\Delta\\n/2<h_k\le n}}k p(n-h_k)\label{ps3}
\end{align}
by \eqref{eq:rectangle} and \eqref{eq:rectangle-bound}. 

If $h_k\le n/2$, then \eqref{eq:ratio-near} yields
\[
p(n-h_k)\le Cp(n)e^{-\beta(2k+m-1)}
\le Cp(n)e^\beta e^{-2\beta k}.
\]
Thus \begin{align}
 \sum_{\substack{k>K+\Delta\\h_k\le n/2}}kp(n-h_k)\ll p(n)\sum_{k>K+\Delta}k e^{-2\beta k}.\no
\end{align}
For every real $T\ge1$, summing the geometric series and its derivative, with integer lower endpoint $\lfloor T\rfloor+1$, gives
\[
\sum_{k>T}k e^{-2\beta k}
\ll\beta^{-2}(1+\beta T)e^{-2\beta T}
\qquad(0<\beta\le1/2).
\]
Take $T=K+\Delta$. Since
\[
e^{-2\beta T}=e^{-2L}L^{-8}=\beta^2L^{-8},
\qquad 1+\beta T=1+L+4\log L\ll L,
\]
we have $$\sum_{k>K+\Delta}k e^{-2\beta k}\ll L^{-7}$$
Thus 
\begin{align}
 \sum_{\substack{k>K+\Delta\\h_k\le n/2}}kp(n-h_k)	\ll p(n)L^{-7}.\label{ps31}
\end{align}

If $n/2<h_k\le n$, \eqref{eq:ratio-far} gives $p(n-h_k)\le Cp(n)e^{-c\sqrt n}$ . 
Then 
\begin{align}
\sum_{\substack{k>K+\Delta\\n/2<h_k\le n}}k p(n-h_k)\le Cp(n)e^{-c\sqrt n}\sum_{\substack{k>K+\Delta\\n/2<h_k\le n}}k =O\left(n^2p(n)e^{-c\sqrt n}\right)\label{ps32}
\end{align}
since the condition $h_k\le n$ implies $k\le(n+1)/2$. 

Substituting \eqref{ps31} and \eqref{ps32} into \eqref{ps3} yields 
\begin{align}\label{s3}
	S_3&\ll p(n)L^{-7}+n^2p(n)e^{-c\sqrt n}=o(p(n)).
\end{align}

Substituting \eqref{s1}, \eqref{s2} and \eqref{s3} into \eqref{s123} gives \begin{align*}
	\left|\NT(m,n)-KN(m,n)\right|
=o(p(n))+O_A\!\bigl(p(n)\log L\bigr).
\end{align*}
 This proves \eqref{eq:rank-concentration}.
\end{proof}

We are now in a position to prove \eqref{eq:NT-main}.
\begin{proof}[Proof of \eqref{eq:NT-main}]
First suppose $0\le m\le A\sqrt n$. Substituting \eqref{eq:rank-count} into Proposition~\ref{prop:rank-concentration}, we obtain
\begin{align}
\NT(m,n)
&=\frac{p(n)L_n}{4}\sech^2\!\left(\frac{\beta_nm}{2}\right)
\left(1+O_A\left(n^{-1/2}\right)\right)
+O_A\!\bigl(p(n)\log L_n\bigr)\no\\
&=\frac{p(n)L_n}{4}\sech^2\!\left(\frac{\beta_nm}{2}\right)
+O_A\!\bigl(p(n)\log L_n\bigr)\label{eq:NT-main1}
\end{align}
since $\sech^2(\beta_nm/2)\le1$ and
$L_n n^{-1/2}=o(\log L_n)$.

For $m=-r<0$, consider conjugation of Ferrers diagrams. A partition of rank $r$ with $k$ parts has largest part $k+r$; its conjugate has rank $-r$ and $k+r$ parts. Consequently,
\begin{equation}\no
N(-r,n)=N(r,n),\qquad
\NT(-r,n)-\NT(r,n)=rN(r,n).
\end{equation}
For $r\le A\sqrt n$, Lemma~\ref{lem:counts} gives
\[
rN(r,n)\ll_A r\beta_np(n)\ll_A p(n).
\]
Since $\log L_n\to\infty$, this $O_A(p(n))$ term is absorbed into the error in \eqref{eq:NT-main1}. Since $\sech^2$ is even, \eqref{eq:NT-main1} follows uniformly throughout $|m|\le A\sqrt n$.

Finally,
\[
 L_n=\frac12\log n+\log\frac{\sqrt6}{\pi},
\qquad
\log L_n=\log\log n+O(1).
\]
These identities together with \eqref{eq:NT-main1} give \eqref{eq:NT-main}.
\end{proof}

\section*{Acknowledgements}
This work was supported by the National Natural Science Foundation of China
(Grant No. 12471315).

\end{document}